\documentclass{article}
\usepackage[cp1251]{inputenc}
\usepackage[russian]{babel}
\usepackage{amssymb,amsfonts,amsmath,amsthm,amscd}
\usepackage{graphicx}
\usepackage{enumerate}
\usepackage{soul}
\usepackage[matrix,arrow,curve]{xy}
\usepackage{longtable}
\usepackage{array}
\usepackage{bm}
\RequirePackage{titling}

\newdimen\symskip
\newdimen\defskip
\newdimen\parind
\parind=\parindent
\newdimen\leftmarge
\newdimen\theoremshape
\topsep\defskip
\makeatletter
\newcommand*{\clei}{\nobreak\hskip\z@skip}

\renewcommand{\:}{\textup{:}}
\renewcommand{\~}{\textup{;}}

\DeclareRobustCommand*{\dh}{\clei\hbox{-}\clei}
\newcommand{\no}{}

\renewcommand{\@listI}{\settowidth\labelwidth{\labheadi{\no}}\listipar{\parind}{\labelwidth}}
\newcommand{\listivpar}{\topsep\defskip\partopsep0pt\parsep-\parskip\itemsep0.5\topsep}
\newcommand{\listipar}[2]{\rightmargin0pt\leftmargin#1\labelsep#1\advance\labelsep-#2\itemindent0pt\listivpar}
\renewcommand{\@listii}{\settowidth\labelwidth{\labheadii{\@roman{\no}}}\listiipar{\parind}{\labelwidth}}
\newcommand{\listiivpar}{\topsep0.5\defskip\partopsep0pt\parsep-\parskip\itemsep0.5\topsep}
\newcommand{\listiipar}[2]{\rightmargin0pt\leftmargin#1\labelsep#1\advance\labelsep-#2\itemindent0pt\listiivpar}
\def\thempfn{\ifcase\value{footnote}1\or *\or **\or ***\else\@ctrerr\fi}
\renewcommand\footnoterule{%
  \kern-3\p@
  \hrule\@width1in
  \kern2.6\p@}
\makeatother

\makeatletter

\renewcommand{\@biblabel}[1]{[#1]}
\renewenvironment{thebibliography}[1]
     {\renewcommand{\refname}{References}%
      \section*{\refname}%
      \@mkboth{\MakeUppercase\refname}{\MakeUppercase\refname}%
      \list{\@biblabel{\@arabic\c@enumiv}}%
           {\itemsep\baselineskip
            \leftmargin\parind
            \settowidth\labelwidth{\@biblabel{#1}}%
            \labelsep\parind\advance\labelsep-\labelwidth
            \@openbib@code
            \usecounter{enumiv}%
            \let\p@enumiv\@empty
            \renewcommand\theenumiv{\@arabic\c@enumiv}}%
      \sloppy
      \clubpenalty4000
      \@clubpenalty\clubpenalty
      \widowpenalty4000%
      \sfcode`\.\@m}
     {\def\@noitemerr
       {\@latex@warning{Empty `thebibliography' environment}}%
      \endlist}

\def\@maketitle{%
  \newpage
  \vskip0.5em%
  UDK \udk%
  \vskip0.5em%
  MSC \msc%
  \vskip1.5em%
  \begin{center}\bf%
  \let\footnote\thanks%
   {\Large\@author\par}%
   \vskip1em%
   {\LARGE\@title\par}%
   \vskip1em%
   {\large\@date}%
  \end{center}%
  \par
  \vskip1.5em}

\def\@title{\@latex@warning@no@line{No \noexpand\title given}}

\renewcommand{\thanksmarkseries}[1]{%
  \def\@bsmarkseries{\renewcommand{\thefootnote}{\thempfn}}}
\thanksmarkseries{arabic}

\makeatother

\makeatletter

\renewcommand\sectionmark[1]{%
 \markright{%
  \ifnum \c@secnumdepth >\z@
   \thesection. \ %
  \fi
 #1}}%

\renewcommand{\section}{\@startsection{section}{1}{0pt}%
{5.5ex plus .5ex minus .2ex}{1.5ex plus .3ex}%
{\center\normalfont\Large\bfseries\sffamily\bom}}
\renewcommand{\subsection}{\@startsection{subsection}{2}{0pt}%
{4.5ex plus .4ex minus .2ex}{0.75ex plus .2ex}%
{\center\normalfont\large\bfseries\sffamily\bom}}
\renewcommand{\subsubsection}{\@startsection{subsubsection}{3}{0pt}%
{2.5ex plus .5ex minus .2ex}{1ex plus .2ex}%
{\center\normalfont\bfseries\sffamily\bom}}

\def\@postskip@{\hskip.5em\relax}

\def\postsection{.\@postskip@}

\def\postsubsection{.\@postskip@}

\def\postsubsubsection{.\@postskip@}

\def\postparagraph{.\@postskip@}

\def\postsubparagraph{.\@postskip@}

\def\@seccntformat#1{\csname pre#1\endcsname\csname the#1\endcsname\csname post#1\endcsname}

\makeatother

\renewcommand{\thesection}{\textup{\arabic{section}}}

\newcommand{\parr}{\par\addvspace{\defskip}}
\newcommand{\theo}[2]{\newtheorem{#1}{#2}}
\newcommand{\deff}[2]{\newenvironment{#1}{\parr\textbf{#2.}}{\parr}}
\theo{cas}{Case}
\theo{problem}{Problem}
\theo{theorem}{Theorem}
\theo{lemma}{Lemma}
\theo{prop}{Proposition}
\theo{stm}{Statement}
\theo{imp}{Corollary}
\theo{ex}{Example}
\deff{df}{Definition}
\deff{note}{Remark}
\deff{denote}{Notation}
\deff{denotes}{Notations}
\deff{hint}{Hint}
\deff{answer}{Answer\:}

\makeatletter
\def\@begintheorem#1#2[#3]{%
  \deferred@thm@head{\the\thm@headfont \thm@indent
    \@ifempty{#1}{\let\thmname\@gobble}{\let\thmname\@iden}%
    \@ifempty{#2}{\let\thmnumber\@gobble}{\let\thmnumber\@iden}%
    \@ifempty{#3}{\let\thmnote\@gobble}{\let\thmnote\@iden}%
    \thm@notefont{\bfseries\upshape}%
    \indent%
    \thm@swap\swappedhead\thmhead{#1}{#2}{#3}%
    \the\thm@headpunct
    \thmheadnl 
    \hskip\thm@headsep
  }%
  \ignorespaces}
\renewenvironment{proof}{\setcounter{cas}{0}\parr\pushQED{\qed}\normalfont$\square\quad$}{\setcounter{cas}{0}\popQED\@endpefalse\parr}

\makeatother

\newcommand{\labheadi}[1]{\textup{#1)}}
\newcommand{\labheadii}[1]{\textup{(#1)}}
\newcommand{\labhi}[1]{\labheadi{\arabic{#1}}}

\newenvironment{nums}[1]{\renewcommand{\no}{#1}\begin{enumerate}}{\end{enumerate}}

\newcommand{\equ}[1]{\begin{equation*}#1\end{equation*}}

\makeatletter

\def\LT@makecaption#1#2#3{%
  \LT@mcol\LT@cols c{\hbox to\z@{\hss\parbox[t]\LTcapwidth{%
    \sbox\@tempboxa{#1{#2. }#3}%
    \ifdim\wd\@tempboxa>\hsize
      #1{#2. }#3%
    \else
      \hbox to\hsize{\hfil\box\@tempboxa\hfil}%
    \fi
    \endgraf\vskip\baselineskip}%
  \hss}}}

\@addtoreset{equation}{section}
\@addtoreset{footnote}{section}
\@addtoreset{footnote}{page}

\newenvironment{casks}{%
  \matrix@check\casks\env@casks
}{%
  \endarray\right.%
}
\def\env@casks{%
  \let\@ifnextchar\new@ifnextchar
  \left\lbrack
  \def\arraystretch{1.2}%
  \array{@{}l@{\quad}l@{}}%
}

\newcounter{numt}
\newcounter{col}
\newcounter{coll}

\makeatother

\renewcommand{\le}{\leqslant}

\newcommand{\fa}{\,\forall\,}

\newcommand{\es}{\varnothing}

\newcommand{\subs}{\subset}

\newcommand{\cln}{\colon}
\newcommand{\nl}{\lhd}

\newcommand{\Ra}{\Rightarrow}
\newcommand{\Lra}{\Leftrightarrow}
\newcommand{\thra}{\twoheadrightarrow}

\newcommand{\ol}{\overline}

\newcommand*{\bw}[1]{#1\nobreak\discretionary{}{\hbox{$\mathsurround=0pt #1$}}{}}
\newcommand{\sco}{,\ldots,}

\newcommand{\ssub}{\bw\subs\ldots\bw\subs}

\newcommand{\br}[1]{\bigl(#1\bigr)}
\newcommand{\Br}[1]{\Bigl(#1\Bigr)}

\newcommand{\bs}[1]{\bigl[#1\bigr]}

\newcommand{\bc}[1]{\bigl\{#1\bigr\}}
\newcommand{\BC}[1]{\Bigl\{#1\Bigr\}}

\newcommand{\mfr}{\mathfrak}

\newcommand{\ggt}{\mfr{g}}
\newcommand{\hgt}{\mfr{h}}
\newcommand{\agt}{\mfr{a}}

\newcommand{\zgt}{\mfr{z}}

\newcommand{\la}{\lambda}

\newcommand{\rh}{\rho}

\DeclareMathOperator{\Ker}{Ker}

\DeclareMathOperator{\Der}{Der}

\DeclareMathOperator{\ad}{ad}

\DeclareMathOperator{\Img}{Im}

\DeclareMathOperator{\tr}{tr}

\DeclareMathOperator{\cha}{char}

\DeclareMathOperator{\rad}{rad}

\DeclareMathOperator{\Spec}{Spec}

\newcommand{\glg}{\mfr{gl}}
\newcommand{\slg}{\mfr{sl}}

\newcommand{\bom}{\boldmath}

\begin{document}

\author{Styrt O.\,G.\thanks{Russia, MIPT, oleg\_styrt@mail.ru}}
\title{Elements of a~Lie algebra\\
acting nilpotently\\
in all its representations}
\date{}
\newcommand{\udk}{512.815.1+512.815.6}
\newcommand{\msc}{17B10+17B20+17B30+17B40}

\maketitle

An equivalent condition for an element of a~Lie algebra acting nilpotently in all its representations is obtained. Namely, it should belong to the
derived algebra and go via factoring over the radical to a~nilpotent element of the corresponding (semisimple) quotient algebra.

\smallskip


\section{Introduction}\label{introd}

Let $\ggt$ be a~Lie algebra over an algebraically closed field of characteristic zero. In the case of a~semisimple Lie algebra~$\ggt$
\begin{itemize}
\item its element~$a$ is called \textit{nilpotent} if the adjoint operator $\ad a$ is nilpotent\~
\item denote the set of all its nilpotent elements by~$N_{\ggt}$.
\end{itemize}
In the general case, denote by~$\hgt$ the semisimple Lie algebra $\ggt/(\rad\ggt)$ and by~$\pi$ the factoring homomorphism $\ggt\thra\hgt$. Here
is the main result of the paper.

\begin{theorem}\label{main} For any element $a\in\ggt$, the following conditions are equivalent\:
\begin{nums}{-1}
\item\label{nilp} in each representation of~$\ggt$, the element~$a$ acts nilpotently\~
\item\label{cond} $a\in[\ggt,\ggt]$ and $\pi(a)\in N_{\hgt}$.
\end{nums}
\end{theorem}

\section{Notations and auxiliary facts}\label{facts}

In this section, a~number of auxiliary notations and statements is given.

First of all, for brevity, the following notations will be used\:
\begin{itemize}
\item $k$ is the main field\~
\item $E_V$ is the identity operator in a~space~$V$ (if the space is clear from the context, then the index can be omitted)\~
\item $\Spec X$ is the set of all eigenvalues of an operator~$X$\~
\item $V_{\la}(X)$ (resp. $V^{\la}(X)$) is the eigenspace (resp. the root subspace) of an operator~$X$ in a~space~$V$ with an eigenvalue~$\la$\~
\item $\zgt(a)$ is the centralizer of an element~$a$ of a~Lie algebra.
\end{itemize}

Let $A$ be an arbitrary algebra (without any requirements to the bilinear operation of multiplication). It is known (and can be proved by elementary
check) that
\begin{itemize}
\item $\Der A$ is a~Lie subalgebra in $\glg(A)$\~
\item for all $d\in\Der A$ and $\la,\mu\in k$, we have $\br{A_{\la}(d)}\cdot\br{A_{\mu}(d)}\subs A_{\la+\mu}(d)$.
\end{itemize}

\begin{lemma}\label{lrn} If $n:=\dim A$, $k=\ol{k}$ and $\cha k\notin(0;n]$, then, for each $d\in\Der A$, $\la\in k^*$ and $a\in A^{\la}(d)$,
the operators of left and right shifts by~$a$ are nilpotent.
\end{lemma}

\begin{proof} By condition, $|\Spec d|\le n$. Further, let $X$ be any of the operators $l_a$ and~$r_a$. Once $\mu\in k$, then
$X\br{A^{\mu}(d)}\subs A^{\la+\mu}(d)$ and the subset $\{m\la+\mu\cln m=0\sco n\}\subs k$ of order $n+1$ is thus not contained in $\Spec d$ that
implies $X^n\br{A^{\mu}(d)}=0$. Hence, $X^n=0$.
\end{proof}

Assume that $k=\ol{k}$ and $\cha k=0$.

Let $\hgt$ be a~semisimple Lie algebra. Fix the non-degenerate symmetric bilinear Killing form $(\cdot,\cdot)$ on~$\hgt$. For any $a\in\hgt$, the
operator $d:=\ad a$ is skew-symmetric and, consequently, $(\Ker d)^{\perp}=\Img d$, i.\,e. $\br{\zgt(a)}^{\perp}=[a,\hgt]$.

\begin{lemma}\label{nort} For each $e\in N_{\hgt}$, we have $e\perp\zgt(e)$.
\end{lemma}

\begin{proof} If $a\in\hgt$ and $[a,e]=0$, then $[\ad a,\ad e]=0$ implying (see Lemma in \cite[sec.\,II.8.2, p.\,36]{Hum})
$(a,e)=\tr\br{(\ad a)(\ad e)}=0$.
\end{proof}

\begin{theorem}\label{coni} Let $e\in\hgt$ be an arbitrary element. Then $(e\in N_{\hgt})\Lra\br{e\in[e,\hgt]}$.
\end{theorem}

\begin{proof} If $e\in N_{\hgt}$, then, according to Lemma~\ref{nort}, $e\in\br{\zgt(e)}^{\perp}=[e,\hgt]$.

Conversely, if $a\in\hgt$ and $e=[a,e]$, then $d:=\ad a\in\Der\hgt$ and $e\in\hgt_1(d)$ that implies (by Lemma~\ref{lrn}) nilpotence of the operator $\ad e$.
\end{proof}

Let $V$ be an arbitrary subspace. It is well-known that $\slg(V)$ is a~simple Lie algebra and the set of all nilpotent operators $V\to V$
coincides with its subset~$N_{\slg(V)}$. For each operator $X\cln V\to V$, we have $\bs{X,\glg(V)}=\bs{X,\slg(V)}\subs\slg(V)$
and, by Theorem~\ref{coni}, nilpotence of~$X$ is equivalent to the inclusion $X\in\bs{X,\glg(V)}$.

Take any Lie algebra~$\ggt$ and its representation~$\rh$ in a~space~$V$.

Let $\agt\subs\ggt$ be an arbitrary subspace. Define \textit{the weight subspaces}
\equ{
V_{\xi}(\agt):=\BC{v\in V\cln av=\br{\xi(a)}\cdot v\,\fa a\in\agt}\subs V\quad(\xi\in\agt^*)}
and, besides, \textit{the set of $\agt$\dh weights} $P_{\agt}:=\bc{\xi\in\agt^*\cln V_{\xi}(\agt)\ne0}\subs\agt^*$.

\begin{theorem}\label{lid}\
\begin{nums}{-1}
\item If the Lie algebra~$\ggt$ is solvable, then $P_{\ggt}\ne\es$.
\item For any ideal $\agt\nl\ggt$ and linear function $\xi\in P_{\agt}$, we have $[\ggt,\agt]\subs\Ker\xi\subs\agt$.
\end{nums}
\end{theorem}

\begin{proof} Both statements are proved in \cite[sec.\,I.\textit{V}.5]{Serr} (formulated on p.\,36 as <<Theorem~1'>> and <<The Main Lemma>> respectively).
\end{proof}

\begin{imp}\label{inv} For any ideal $\agt\nl\ggt$ and function $\xi\in\agt^*$, we have $\ggt\br{V_{\xi}(\agt)}\subs V_{\xi}(\agt)$.
\end{imp}

\begin{proof} Assume that $\xi\in P_{\agt}$ (otherwise, it is nothing to prove). Let $b\in\ggt$ and $v\in V_{\xi}(\agt)$ be arbitrary elements. According to
Theorem~\ref{lid}, for each $a\in\agt$, the element $a_0:=[b,a]\in\agt$ belongs to $\Ker\xi$ implying $a_0v=\br{\xi(a_0)}\cdot v=0$,
$abv=bav=b\Br{\br{\xi(a)}\cdot v}=\br{\xi(a)}\cdot(bv)$. So, we obtained that $bv\in V_{\xi}(\agt)$.
\end{proof}

\begin{imp} If the representation~$\rh$ is irreducible, then $\rh(\rad\ggt)\subs kE_V$.
\end{imp}

\begin{proof} The ideal $\agt:=\rad\ggt\nl\ggt$ is solvable. Hence (see Theorem~\ref{lid}), for some function $\xi\in\agt^*$, the subspace
$V_{\xi}(\agt)\subs V$ is nonzero. By Corollary~\ref{inv}, it is $\ggt$\dh invariant and, thus, equals~$V$.
\end{proof}

\begin{imp}\label{caze} If the representation~$\rh$ is irreducible, then $\rh\br{(\rad\ggt)\cap[\ggt,\ggt]}=0$.
\end{imp}

\section{Proofs of the results}\label{prove}

In this section, Theorem~\ref{main} is proved.

Return to notations and agreements of \S\,\ref{introd}. Recall that the main field~$k$ is assumed to be algebraically closed and to have characteristic
zero.

\ul{$\text{\ref{nilp}}\Ra\text{\ref{cond}}$}.

Suppose that the element $a\in\ggt$ acts nilpotently in any representation of~$\ggt$.

Consider an arbitrary ideal $\agt\nl\ggt$, the factoring homomorphism $\pi_{\agt}\cln\ggt\thra\ggt/\agt$ and, besides, the element
$a_{\agt}:=\pi_{\agt}(a)\in\ggt/\agt$. Each representation $\rh\cln\ggt/\agt\to\glg(V)$ naturally induces the representation
$(\rh\circ\pi_{\agt})\cln\ggt\to\glg(V)$ and, therefore, the operator $\rh(a_{\agt})=(\rh\circ\pi_{\agt})(a)\in\glg(V)$ is nilpotent.
Research two special particular cases.

\begin{nums}{-1}\renewcommand{\labhi}[1]{\labheadii{\roman{#1}}}
\item $\agt:=[\ggt,\ggt]$, $V\ne0$. In this case, the Lie algebra $\ggt/\agt$ is commutative. For any function $\xi\in(\ggt/\agt)^*$, we have
a~Lie algebra homomorphism $\rh\cln\ggt/\agt\to\glg(V),\,x\to\br{\xi(x)}\cdot E$ and, hence, the operator
$\br{\xi(a_{\agt})}\cdot E=\rh(a_{\agt})\in\glg(V)$ is nilpotent implying $\xi(a_{\agt})=0$. So, the element
$\pi_{\agt}(a)=a_{\agt}\in\ggt/\agt$ is trivial\~ in other words, $a\in\agt=[\ggt,\ggt]$.
\item $\agt:=\rad\ggt$. Then $\ggt/\agt$ is the semisimple Lie algebra~$\hgt$, furthermore, $\pi_{\agt}=\pi$ and $a_{\agt}=\pi(a)$. So,
$\ad(a_{\agt})\in\glg(\hgt)$ is a~nilpotent operator, i.\,e. $\pi(a)=a_{\agt}\in N_{\hgt}$.
\end{nums}

\ul{$\text{\ref{cond}}\Ra\text{\ref{nilp}}$}.

Now, suppose that $a\in[\ggt,\ggt]$ and $\pi(a)\in N_{\hgt}$.

By Theorem~\ref{coni}, $\pi(a)\in\bs{\pi(a),\hgt}=\bs{\pi(a),\pi(\ggt)}=\pi\br{[a,\ggt]}$,
\equ{
a\in\br{[a,\ggt]+(\rad\ggt)}\cap[\ggt,\ggt]=[a,\ggt]+\br{(\rad\ggt)\cap[\ggt,\ggt]}.}
If $\rh\cln\ggt\to\glg(V)$ is an arbitrary irreducible representation, then, according to Corollary~\ref{caze},
$\rh(a)\in\rh\br{[a,\ggt]}=\bs{\rh(a),\rh(\ggt)}\subs\bs{\rh(a),\glg(V)}$ and, hence, the operator $\rh(a)\cln V\to V$ is nilpotent. Thus, in each
irreducible representation of the Lie algebra~$\ggt$, its element~$a$ acts nilpotently. This statement still holds without requiring irreducibility\:
for any representation $\ggt\to\glg(V)$, there exists an increasing chain of subrepresentations $0=V_0\ssub V_n=V$ such that all quotient representations
$V_m/V_{m-1}$ ($m=1\sco n$) are irreducible.

So, Theorem~\ref{main} is completely proved.

\section*{Acknowledgements}

The author is grateful to Prof. \fbox{E.\,B.\?Vinberg} for exciting an interest to fundamental algebraic science.

The author dedicates the article to E.\,N.\?Troshina.

\end{document}